\documentclass[10pt,letterpaper,oneside]{amsart}

  \usepackage{amsmath,amssymb,amsthm}
  \usepackage{amscd}
  \usepackage{geometry}
  \usepackage{graphicx,epstopdf}
  \usepackage{color}
  \usepackage{enumerate}
  \usepackage{xspace}
  \usepackage{tipa}
  \usepackage{epsfig}
  \usepackage{pinlabel}
  \usepackage[all]{xy}

  \usepackage{hyperref}

  \newtheorem{theorem}{Theorem}[section]

  \newtheorem{lemma}[theorem]{Lemma}
  
  \newtheorem*{conjecture*}{Conjecture}

  \theoremstyle{definition}

  \newtheorem*{claim*}{Claim}

  \newtheorem*{question*}{Question}
  \newtheorem*{answer*}{Answer}
  \newtheorem*{application*}{Application}

  \theoremstyle{remark}
  
  \newtheorem*{remark*}{Remark}

  \DeclareMathOperator{\Mod}{Mod}

  \newcommand{\param}{{\mathchoice{\mkern1mu\mbox{\raise2.2pt\hbox{$
  \centerdot$}}
  \mkern1mu}{\mkern1mu\mbox{\raise2.2pt\hbox{$\centerdot$}}\mkern1mu}{
  \mkern1.5mu\centerdot\mkern1.5mu}{\mkern1.5mu\centerdot\mkern1.5mu}}}

\begin{document}


\title {An arc graph distance formula for the flip graph
 }

  \author   {Funda G\"ultepe}
  \author {Christopher J. Leininger}
\address{Department of Mathematics\\
University of Illinois at Urbana-Champaign\\ Urbana, IL 61801}
\email{fgultepe@illinois.edu}
\urladdr{http://www.math.illinois.edu/~fgultepe/}

\address{Department of Mathematics\\
University of Illinois at Urbana-Champaign\\ Urbana, IL 61801}
\email{clein@math.uiuc.edu}
\urladdr{http://www.math.uiuc.edu/~clein/}

\begin{abstract} Using existing technology, we prove a Masur-Minsky style distance formula for flip-graph distance between two triangulations, expressed as a sum of the distances of the projections of these triangulations into arc graphs of the suitable subsurfaces of $S$.
\end{abstract}

\maketitle

\section{introduction}

Let $S$ be a surface with at least one puncture and $\chi(S) < 0$, and write $\mathcal F(S)$ for the {\em flip graph of $S$}.  This is the graph whose vertices are in a one-to-one correspondence with ideal triangulations, and whose edges connect triangulations that differ by a {\em flip}; see \cite{disarlo-parlier} and Figure~\ref{fig:flip1}.  The purpose of this note is to prove the following formula estimating distance in $\mathcal F(S)$.

\begin{theorem} \label{T:main}
Fix $S$, a connected, orientable, finite type, surface of non-positive Euler characteristic, with at least one puncture, and not a pair of pants.
For any $k > 0$ sufficiently large, there exists $K\geq 1, C \geq 0$ so that for any two triangulations $T_1,T_2 \in \mathcal F(S)$ we have
\[ d_{\mathcal F}(T_1,T_2) \stackrel{K,C}{\asymp} \sum_{Y \subseteq S} [d_{\mathcal A(Y)}(T_1,T_2)]_k. \]
\end{theorem}
The distances on the right are {\em arc graph} distances in subsurfaces, $[x]_k$ is the cut-off function giving value $x$ if $x \geq k$ and $0$ otherwise, and $x \stackrel{K,C}{\asymp} y$ is shorthand for the condition $\frac{1}{K}(x - C) \leq y \leq Kx + C$.   See the next section for a precise statement.

Our theorem follows more-or-less directly from the Masur-Minsky distance formula \cite{MM2} and the Masur-Schleimer distance formula \cite{MSch1}, but seems worth making explicit since $\mathcal F(S)$ is an important, particularly tractable, geometric model for the mapping class group of $S$ (see e.g.~\cite{harer-stab,harer-virt,Hatcher-triangulation,moshAut,disarlo-parlier,bell}), while on the other side, the geometry of the arc graph has been greatly simplified in \cite{HPW}.
Various distance formulas \cite{MM1,MSch1,Rafi2} have been used extensively to understand the geometry of mapping class group, Teichm\"uller space, and homomorphisms  (see e.g.\cite{brock,Behr,kentleininger,BowBundle,BDS,BBML,CLM,tao,EMR,BBF1}) and have motivated research in related areas (see e.g. \cite{SS,CP2,ST2,sisto,kimkoberda,BestSF,ST1,HamHen,vogt}).

It would be interesting to find a proof of Theorem~\ref{T:main} that does not appeal to the previous distance formulas.



\bigskip

\noindent
{\bf Acknowledgements.}  The authors thank Valentina Disarlo, Hugo Parlier, and Kasra Rafi for useful conversations.  The second author was partially supported by NSF grant DMS 1510034.

\section{The proof.}

For a surface $S$ of genus $g$ with $n$ punctures, we write $\xi(S) = 3g-3+n$ (we do not distinguish between a puncture and a hole, and will only refer to punctures to avoid confusion later).  All surfaces we consider are orientable, have at least one puncture, and have $\xi >0$, with one exception: we allow annuli (which have $\xi = -1$).  In particular, we exclude three-punctured spheres in all of what follows.  Arcs, curves, multiarcs, and multicurves are assumed essential and are considered up to isotopy.  Multiarcs and multicurves have pairwise non-isotopic components.  Ideal triangulations are multiarcs with a maximal number of components.  Markings are complete clean markings (see \cite{MM2}).

We write $\mathcal C(Y)$ for the arc-and-curve graph of a surface $Y$, which is quasi-isometric to the curve graph (more precisely, the inclusion of the curve graph into the arc-and-curve graph is a quasi-isometry).   Given any multiarc, multicurve, marking, or triangulation, $\alpha$ on a surface $S$ and subsurface $Y \subseteq S$ which is not an annulus, we let $\pi_Y(\alpha)$ denote the arc-and-curve projection:  This is the union of the isotopy classes of arcs and curves of intersection of $\alpha$ with $Y$ (assuming they are in minimal position).  For $Y$ an annulus, we use the usual projection to $\mathcal A(Y)$ via the cover corresponding to $Y$; see \cite{MM2} for details.    We will write
\[ d_{\mathcal C(Y)}(\alpha,\beta) =  diam( \pi_Y(\alpha) \cup \pi_Y(\beta))\]
where the diameter is taken in $\mathcal C(Y)$.    When the projections are non-empty, for example if $\alpha$ is a marking or a triangulation, then $d_{\mathcal C(Y)}$ satisfies a triangle inequality.  If $\alpha$ is an arc or a triangulation, then $\pi_Y(\alpha)$ is in the arc graph, $\mathcal A(Y)$, and so we can define $d_{\mathcal A(Y)}(\alpha,\beta)$ similarly.  We note that using the arc-and-curve graph projection, it follows that for any $X \subseteq Y \subseteq S$, we have $\pi_X \circ \pi_Y = \pi_X$, unless $X$ is an annulus.

As stated in the introduction, the flip graph $\mathcal F(S)$ is the graph whose vertex set is the set isotopy classes of (ideal) triangulations.
Two vertices in the graph share an edge if they are related by a {\em flip}, in other words, if they differ at most by an arc; see  \cite{disarlo-parlier} and Figure \ref{fig:flip1}.

\begin{figure}
\centering{\includegraphics[width=0.6\textwidth]{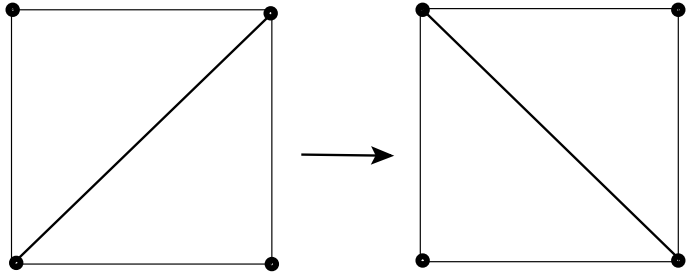}}
\caption{An example of a flip in the flip graph}
\label{fig:flip1}
\end{figure}

For markings $\mu_1,\mu_2$ on $S$, we let $d_{\mathcal M}(\mu_1,\mu_2)$ denote the distance in the marking graph $\mathcal M(S)$; see \cite{MM2}.  The first distance formula we will need is due to Masur and Minsky:
\begin{theorem}[\cite{MM2}] \label{T:mm-dist} Fix $S$, a connected, orientable surface with $\xi(S) > 0$.  For any $k > 0$ sufficiently large, there exists $K,C \geq 1$ so that for any two markings $\mu_1,\mu_2$ we have
\[ d_{\mathcal M}(\mu_1,\mu_2) \stackrel{K,C}{\asymp} \sum_{Y \subseteq S} [d_{\mathcal C(Y)}(\mu_1,\mu_2)]_k.\]
\end{theorem}
In this theorem, we note that $K,C$ can be chosen to depend monotonically on $k$.  Indeed, the right-hand side becomes less efficient at estimating the left-hand side as $k$ increases, so at least coarsely, this monotonicity is necessary.

There is a distance formula for arc graphs due to Masur and Schleimer (see Lemma 7.2 and Theorems 5.10 and 13.1 of \cite{MSch1}).  To state this formula, we recall that given a surface $Y$, a {\em hole} for $\mathcal A(Y)$ is an essential subsurface $X \subseteq Y$ such that the punctures of $Y$ are also punctures of $X$, which we write as $\partial Y \subseteq \partial X$.  We let $H(\mathcal A(Y))$ denote the set of holes for $\mathcal A(Y)$.  For $Y$ an annulus, the only hole for $\mathcal A(Y)$ is $Y$, and $Y$ is not a hole for $\mathcal A(X)$, for any other surface $X$.

\begin{theorem}[\cite{MSch1}] \label{T:msch-dist} Fix $S$, a connected, orientable surface with at least one puncture and $\xi(S) > 0$.
Then for any $k > 0$ sufficiently large, there exists $K \geq 1, C \geq 0$ so that for any two arcs $\alpha_1,\alpha_2$,
\[ d_{\mathcal A(S)}(\alpha_1,\alpha_2) \stackrel{K,C}{\asymp} \sum_{X \in H(\mathcal A(S))} [d_{\mathcal C(X)}(\alpha_1,\alpha_2)]_k .\]
\end{theorem}

The proof of Theorem~\ref{T:main} also requires the following elementary observation.
\begin{lemma} \label{L:where it's a hole}
Fix a surface $S$.  For any essential subsurface $X \subseteq S$, there are at most $2^{\xi(X)} \leq 2^{\xi(S)}$ subsurfaces $Y$ such that $X$ is a hole for $\mathcal A(Y)$.
\end{lemma}
\begin{proof}
An essential subsurface $X$ is a component of the complement of an essential multicurve that we denote $\partial_0 Y$.
If $X$ is a hole for $\mathcal A(Y)$, then observe that $Y$ is the component of the complement of $\partial_0 Y$ containing $X$.  Therefore $Y$ is determined by $X$ and the multicurve $\partial_0 Y \subseteq \partial_0 X$.  There are $2^{|\partial_0 X|}$ submulticurves of $\partial_0 X$, and $|\partial_0X| \leq \xi(X)$, and hence at most this many $Y \subseteq S$ such that $X$ is a hole for $\mathcal A(Y)$.
\end{proof}

\begin{proof}[Proof of Theorem~\ref{T:main}]
Fix $S$.  For every ideal triangulation $T$, we choose a marking $\mu(T)$ so that $i(T,\mu(T))$ is minimized (here we simply take the sum of intersection numbers of components of $T$ and $\mu(T)$).   Because the mapping class group $\Mod(S)$ has only finitely many orbits on $\mathcal F(S)$, this intersection number is uniformly bounded, independent of $T$.  Consequently, there exists  $\delta_0 > 0$ such that for each triangulation $T$ of $S$ and every subsurface $Y \subseteq S$ we have
\begin{equation} \label{eq:mu(T) and T bdd}
d_{\mathcal C(Y)}(\mu(T),T) < \delta_0.
\end{equation}

Furthermore, we claim that $T \mapsto \mu(T)$ is coarsely $\Mod(S)$--equivariant.  More precisely, for every $g \in \Mod(S)$ and $T \in \mathcal F(S)$, we claim that $d_{\mathcal M}(\mu(gT),g\mu(T))$ is uniformly bounded.  This follows from Theorem~\ref{T:mm-dist} since \eqref{eq:mu(T) and T bdd} and the triangle inequality imply
\begin{eqnarray*}
d_{\mathcal C(Y)}(\mu(gT),g\mu(T)) & \leq & d_{\mathcal C(Y)}(\mu(gT),gT) + d_{\mathcal C(Y)}(gT,g\mu(T)) \\
&= & d_{\mathcal C(Y)}(\mu(gT),gT) + d_{\mathcal C(g^{-1}Y)}(T,\mu(T)) \leq 2\delta_0 \end{eqnarray*}
Since $\Mod(S)$ acts cocompactly by isometries on the proper geodesic spaces $\mathcal F(S)$ and $\mathcal M(S)$, the Milnor-$\check{S}$varc Lemma implies $T \mapsto \mu(T)$ is a quasi-isometry.
Thus, for $T_1,T_2 \in \mathcal F(S)$ and $\mu_i = \mu(T_i)$, for $i = 1,2$ we have
\begin{equation}  \label{eq:flip to marking} d_{\mathcal F}(T_1,T_2) \asymp d_{\mathcal M}(\mu_1,\mu_2).\end{equation}
Let $(K_0,C_0)$ be the implicit constants in this coarse equation.

Next, we choose constants $0 < k_1 < k_2 < k_3 < \infty$ large enough so that for all $T_1,T_2 \in \mathcal F(S)$:\\
{\bf (i)} If $X$ is a hole for $\mathcal A(Y)$ and $d_{\mathcal C(X)}(T_1,T_2) \geq k_3$, then $d_{\mathcal A(Y)}(T_1,T_2) \geq k_2$; and \\
{\bf (ii)} if $d_{\mathcal A(Y)}(T_1,T_2) \geq k_2$, then
\[ d_{\mathcal A(Y)}(T_1,T_2) \asymp  \sum_{X \in H(\mathcal A(Y))} [d_{\mathcal C(X)}(T_1,T_2)]_{k_1}\]
where the implicit constants in this coarse equation are $(K_1,0)$.  For {\bf (ii)}, this means that when the arc graph distance is at least $k_2$, the sum with cut-off function $k_1$ is correct with only a multiplicative error.  To see that we can find such $k_1,k_2,k_3$ and $K_1$, we first appeal to Theorem~\ref{T:msch-dist} to find $k_1, k_2, K_1$ so that {\bf (ii)} holds.  This is possible since once the the arc-graph distance is bigger than twice the additive constant, say, then by doubling the multiplicative constant, we may remove the additive error.  Appealing to Theorem~\ref{T:msch-dist} again guarantees that for $k_3$ sufficiently large {\bf (i)} also holds.  For reasons that will become clear later, we will also assume that $k_1 \geq 10 \delta$ and that $k_1-2 \delta_0$ is above the threshold for Theorem~\ref{T:mm-dist} to hold.

For $T_1,T_2 \in \mathcal F(S)$, let $\Omega(T_1,T_2,k_2)$ be the set of subsurfaces $Y \subseteq S$ so that $d_{\mathcal A(Y)}(T_1,T_2) \geq k_2$.  Then we have
\[ \sum_{Y \subseteq S} [d_{\mathcal A(Y)}(T_1,T_2)]_{k_2} = \sum_{Y \in \Omega(T_1,T_2,k_2)} d_{\mathcal A(Y)}(T_1,T_2) \asymp \sum_{Y \in \Omega(T_1,T_2,k_2)} \sum_{X \in H(\mathcal A(Y))} [d_{\mathcal C(X)}(T_1,T_2)]_{k_1}.\]
The implicit constants in the coarse equation are again $(K_0,0)$ by {\bf (ii)}.

Let $\mathcal H = \mathcal H(T_1,T_2,k_1,k_2,k_3)$ be the set of all $X$ which appear with nonzero contribution in the sum on the right-hand side of the above coarse equation.  We note that $\mathcal H$ does not keeping track of how many times such an $X$ appears.  By Lemma~\ref{L:where it's a hole}, any $X \in \mathcal H$ appears at most $2^{\xi(S)}$ times in the sum.  Therefore we have
\begin{equation}
\label{eq:arc graphs to H subsurfaces} \sum_{X \in \mathcal H} d_{\mathcal C(X)}(T_1,T_2)  \asymp \sum_{Y \subseteq S} [d_{\mathcal A(Y)}(T_1,T_2)]_{k_2}.
\end{equation}
Here the implicit constants can be taken to be $(2^{\xi(S)}K_0,0)$.

By definition, for each $X \in \mathcal H$, $d_{\mathcal C(X)}(T_1,T_2) \geq k_1$.  On the other hand, if $d_{\mathcal C(X)}(T_1,T_2) \geq k_3$, then $X \in \mathcal H$.   Thus $\mathcal H$ contains {\em all} subsurfaces with distance at least $k_3$ and {\em some} subsurfaces with distance at least $k_1$.  Since $d_{\mathcal C(X)}(\mu_i,T_i) \leq \delta_0$, it follows that if $X \in \mathcal H$, then $d_{\mathcal C(X)}(\mu_1,\mu_2) \geq k_1 - 2\delta_0$, and if $d_{\mathcal C(X)}(\mu_1,\mu_2) \geq k_3 + 2 \delta_0$, then $X \in \mathcal H$.  By the monotonicity of the constants in Theorem~\ref{T:mm-dist}, we have
\begin{equation} \label{eq:marking to H subsurfaces}
d_{\mathcal M}(\mu_1,\mu_2) \asymp \sum_{X \in \mathcal H} d_{\mathcal C(X)}(\mu_1,\mu_2).
\end{equation}
Here the implicit constants $(K_2,C_2)$ in the coarse equation are the same as those in Theorem~\ref{T:mm-dist} for threshold $k_3+2\delta_0$.  Finally, since $k_1 \geq 10 \delta_0$, we have
\begin{equation} \label{eq:T to mu on H}  \sum_{X \in \mathcal H} d_{\mathcal C(X)}(\mu_1,\mu_2)  \asymp  \sum_{X \in \mathcal H} d_{\mathcal C(X)}(T_1,T_2)\end{equation}
and one can check that the implicit constant is $(\frac{9}{8},0)$ (since each term on the left differs from the corresponding term on the right by an additive error which is small compared to it size).

Setting $k = k_2$, and combining \eqref{eq:flip to marking},  \eqref{eq:marking to H subsurfaces} \eqref{eq:T to mu on H}, and \eqref{eq:arc graphs to H subsurfaces}
\[ d_{\mathcal F}(T_1,T_2) \asymp \sum_{Y \subseteq S} [d_{\mathcal A(Y)}(T_1,T_2)]_{k_2}\]
where the implicit constants in the coarse equation depend on all the above constants.  This completes the proof.
\end{proof}

\bibliographystyle{alpha}
\bibliography{main}

\begin{thebibliography}{BKMM12}

\bibitem[BBF15]{BBF1}
Mladen Bestvina, Ken Bromberg, and Koji Fujiwara.
\newblock Constructing group actions on quasi-trees and applications to mapping
  class groups.
\newblock {\em Publ. Math. Inst. Hautes \'Etudes Sci.}, 122:1--64, 2015.

\bibitem[BDS11]{BDS}
Jason Behrstock, Cornelia Dru{\c{t}}u, and Mark Sapir.
\newblock Median structures on asymptotic cones and homomorphisms into mapping
  class groups.
\newblock {\em Proc. Lond. Math. Soc. (3)}, 102(3):503--554, 2011.

\bibitem[Beh06]{Behr}
Jason~A. Behrstock.
\newblock Asymptotic geometry of the mapping class group and {T}eichm\"uller
  space.
\newblock {\em Geom. Topol.}, 10:1523--1578, 2006.

\bibitem[Bel14]{bell}
Mark~C. Bell.
\newblock Recognising mapping classes, 2014.

\bibitem[BF14]{BestSF}
Mladen Bestvina and Mark Feighn.
\newblock Subfactor projections.
\newblock {\em J. Topol.}, 7(3):771--804, 2014.

\bibitem[BKMM12]{BBML}
Jason Behrstock, Bruce Kleiner, Yair Minsky, and Lee Mosher.
\newblock Geometry and rigidity of mapping class groups.
\newblock {\em Geom. Topol.}, 16(2):781--888, 2012.

\bibitem[Bow09]{BowBundle}
Brian~H. Bowditch.
\newblock Atoroidal surface bundles over surfaces.
\newblock {\em Geom. Funct. Anal.}, 19(4):943--988, 2009.

\bibitem[Bro03]{brock}
Jeffrey~F. Brock.
\newblock The {W}eil-{P}etersson metric and volumes of 3-dimensional hyperbolic
  convex cores.
\newblock {\em J. Amer. Math. Soc.}, 16(3):495--535 (electronic), 2003.

\bibitem[CLM12]{CLM}
Matt~T. Clay, Christopher~J. Leininger, and Johanna Mangahas.
\newblock The geometry of right-angled {A}rtin subgroups of mapping class
  groups.
\newblock {\em Groups Geom. Dyn.}, 6(2):249--278, 2012.

\bibitem[CP12]{CP2}
Matt Clay and Alexandra Pettet.
\newblock Relative twisting in outer space.
\newblock {\em J. Topol. Anal.}, 4(2):173--201, 2012.

\bibitem[DP14]{disarlo-parlier}
Valentina Disarlo and Hugi Parlier.
\newblock The geometry of flip graphs and mapping class groups.
\newblock arXiv:1411.4285 [math.GT], 2014.

\bibitem[EMR14]{EMR}
Alex Eskin, Howard Masur, and Kasra Rafi.
\newblock Large scale rank of teichmuller space.
\newblock arXiv:1307.3733v2 [math.GT], 2014.

\bibitem[Har85]{harer-stab}
John~L. Harer.
\newblock Stability of the homology of the mapping class groups of orientable
  surfaces.
\newblock {\em Ann. of Math. (2)}, 121(2):215--249, 1985.

\bibitem[Har86]{harer-virt}
John~L. Harer.
\newblock The virtual cohomological dimension of the mapping class group of an
  orientable surface.
\newblock {\em Invent. Math.}, 84(1):157--176, 1986.

\bibitem[Hat91]{Hatcher-triangulation}
Allen Hatcher.
\newblock On triangulations of surfaces.
\newblock {\em Topology Appl.}, 40(2):189--194, 1991.

\bibitem[HH15]{HamHen}
Ursula Hamenst{\"a}dt and Sebastian Hensel.
\newblock Spheres and projections for {O}ut({$F_n$}).
\newblock {\em J. Topol.}, 8(1):65--92, 2015.

\bibitem[HPW15]{HPW}
Sebastian Hensel, Piotr Przytycki, and Richard C.~H. Webb.
\newblock 1-slim triangles and uniform hyperbolicity for arc graphs and curve
  graphs.
\newblock {\em J. Eur. Math. Soc. (JEMS)}, 17(4):755--762, 2015.

\bibitem[KK14]{kimkoberda}
Sang-Hyun Kim and Thomas Koberda.
\newblock The geometry of the curve graph of a right-angled {A}rtin group.
\newblock {\em Internat. J. Algebra Comput.}, 24(2):121--169, 2014.

\bibitem[KL08]{kentleininger}
Richard~P. Kent, IV and Christopher~J. Leininger.
\newblock Shadows of mapping class groups: capturing convex cocompactness.
\newblock {\em Geom. Funct. Anal.}, 18(4):1270--1325, 2008.

\bibitem[MM99]{MM1}
Howard~A. Masur and Yair~N. Minsky.
\newblock Geometry of the complex of curves. {I}. {H}yperbolicity.
\newblock {\em Invent. Math.}, 138(1):103--149, 1999.

\bibitem[MM00]{MM2}
H.~A. Masur and Y.~N. Minsky.
\newblock Geometry of the complex of curves. {II}. {H}ierarchical structure.
\newblock {\em Geom. Funct. Anal.}, 10(4):902--974, 2000.

\bibitem[Mos95]{moshAut}
Lee Mosher.
\newblock Mapping class groups are automatic.
\newblock {\em Ann. of Math. (2)}, 142(2):303--384, 1995.

\bibitem[MS13]{MSch1}
Howard Masur and Saul Schleimer.
\newblock The geometry of the disk complex.
\newblock {\em J. Amer. Math. Soc.}, 26(1):1--62, 2013.

\bibitem[Raf07]{Rafi2}
Kasra Rafi.
\newblock A combinatorial model for the {T}eichm\"uller metric.
\newblock {\em Geom. Funct. Anal.}, 17(3):936--959, 2007.

\bibitem[Sis13]{sisto}
Alessandro Sisto.
\newblock Projections and relative hyperbolicity.
\newblock {\em Enseign. Math. (2)}, 59(1-2):165--181, 2013.

\bibitem[SS12]{SS}
Lucas Sabalka and Dmytro Savchuk.
\newblock Submanifold projection.
\newblock arXiv:1211.3111v1 [math.GR], 2012.

\bibitem[Tao13]{tao}
Jing Tao.
\newblock Linearly bounded conjugator property for mapping class groups.
\newblock {\em Geom. Funct. Anal.}, 23(1):415--466, 2013.

\bibitem[Tay13]{ST2}
Samuel Taylor.
\newblock Right-angled {A}rtin groups and {O}ut($f_n$) {I}: quasi-isometric
  embeddings.
\newblock arXiv:1303.6889 [math.GT], 2013.

\bibitem[Tay14]{ST1}
Samuel~J. Taylor.
\newblock A note on subfactor projections.
\newblock {\em Algebr. Geom. Topol.}, 14(2):805--821, 2014.

\bibitem[Vog15]{vogt}
Karen Vogtmann.
\newblock On the geometry of outer space.
\newblock {\em Bull. Amer. Math. Soc. (N.S.)}, 52(1):27--46, 2015.

\end{thebibliography}
\end{document}